\documentclass[notitlepage, 11pt]{article}
\usepackage[shortlabels]{enumitem}
\usepackage{amsmath}
\usepackage{amssymb}
\usepackage{amsthm}
\usepackage{abstract}
\usepackage{xcolor}
\usepackage{comment}
\usepackage{mathtools}
\usepackage{graphicx}
\usepackage{caption}
\usepackage{subcaption}
\usepackage{float}
\usepackage{enumitem}
\setlist[itemize]{noitemsep}
\setlist[enumerate]{noitemsep}
\usepackage{hyperref}
\usepackage{cleveref}
\hypersetup{
    colorlinks=true,
    linkcolor=blue,
    filecolor=magenta,      
    urlcolor=cyan
    }
\makeatletter
\newcommand*{\barfix}[2][.175ex]{%
  \mathpalette{\@barfix{#1}}{#2}%
}
\newcommand*{\@barfix}[3]{%
  \vbox{%
    \kern#1\relax
    \hbox{$#2#3\m@th$}%
  }%
}
\makeatother

\newtheorem{theorem}{Theorem}
\newtheorem{thm}{Theorem}[section]

\newtheorem{lemma}[thm]{Lemma}

\newcommand{\footremember}[2]{%
    \footnote{#2}
    \newcounter{#1}
    \setcounter{#1}{\value{footnote}}%
}
\newcommand{\footrecall}[1]{%
    \footnotemark[\value{#1}]%
} 

\usepackage[margin=2cm]{geometry} 

\title{\vspace{-2em}Components, large and small, are as they should be II:\\
supercritical percolation on regular graphs of constant degree}
\author{%
Sahar Diskin \footremember{alley}{School of Mathematical Sciences, Tel Aviv University, Tel Aviv 6997801, Israel. \\Emails: sahardiskin@mail.tau.ac.il, krivelev@tauex.tau.ac.il.}%
\and Michael Krivelevich \footrecall{alley}%
}
\date{}

\begin{document}
\maketitle
\vspace{-2em}
\begin{abstract}
Let $d\ge 3$ be a fixed integer. Let $y\coloneqq y(p)$ be the probability that the root of an infinite $d$-regular tree belongs to an infinite cluster after $p$-bond-percolation. We show that for every constants $b,\alpha>0$ and $1<\lambda< d-1$, there exist constants $c,C>0$ such that the following holds. Let $G$ be a $d$-regular graph on $n$ vertices, satisfying that for every $U\subseteq V(G)$ with $|U|\le \frac{n}{2}$, $e(U,U^c)\ge b|U|$ and for every $U\subseteq V(G)$ with $|U|\le \log^Cn$, $e(U)\le (1+c)|U|$. Let $p=\frac{\lambda}{d-1}$. Then, with probability tending to one as $n$ tends to infinity, the largest component $L_1$ in the random subgraph $G_p$ of $G$ satisfies $\left|1-\frac{|L_1|}{yn}\right|\le \alpha$, and all the other components in $G_p$ are of order $O\left(\frac{\lambda\log n}{(\lambda-1)^2}\right)$. This generalises (and improves upon) results for random $d$-regular graphs.
\end{abstract}

\section{Introduction}
Given a \textit{host graph} $G=(V,E)$ and a probability $p\in [0,1]$, the random subgraph $G_p$ is obtained by performing independent $p$-bond percolation on $G$, that is, we form $G_p$ by retaining each edge in $E$ independently with probability $p$. Perhaps the most well-studied example is the binomial random graph $G(n,p)$, which is equivalent to percolation with probability $p$ on the complete graph $K_n$. For more background on random graphs and on percolation, we refer the interested reader to \cite{B01,BR06,FK16,G99,JLR00,K82}.

A classical result of Erd\H{o}s and R\'enyi from 1960 \cite{ER60} states that $G(n,p)$ undergoes a fundamental phase transition, with respect to its connected component structure, when the expected average degree is around one (that is, $p\cdot (n-1)\approx 1$): in the \textit{subcritical regime}, when the expected average degree is less than one, all the components are typically of order $O(\log n)$, whereas in the \textit{supercritical regime}, when the expected average degree is larger than one, there likely emerges a unique giant component, taking linear (in $n$) fraction of the vertices, and all other components are typically of order $O(\log n)$. In fact, it is known that in the supercritical regime in $G(n,p)$, the asymptotic relative size of the giant component is dictated by the survival probability of a Galton-Watson process with offspring distribution $Bin(n-1, p)$. 

A quantitatively similar behaviour has been observed in percolation on several (families of) $d$-regular graphs, when $d=\omega_n(1)$: pseudo-random $(n,d,\lambda)$-graphs with $\lambda=o(d)$ \cite{FKM04}, the $d$-dimensional hypercube $Q^d$ \cite{AKS81, BKL92}, and other $d$-regular `high-dimensional' graphs \cite{CDE24, DEKK22, L22}. In a companion paper \cite{DK25B}, the authors consider percolation on $d$-regular graphs $G$ with $d=\omega(1)$, and provide sufficient and essentially tight conditions guaranteeing a phase transition in $G_p$ quantitatively similar to that of $G(n,p)$ when the expected average degree is around one. Roughly, it suffices to require $G$ to satisfy a (very) mild `global' edge expansion, and to have a fairly good control on the expansion of small sets in $G$. For the case of growing degree, this serves as a unified proof for the aforementioned `concrete' host graphs.

In this paper, we consider $d$-regular graphs $G$ on $n$ vertices, where $d\ge 3$ is a fixed integer and our asymptotics are in $n$.\footnote{When $d=2$, the graph $G$ is just a collection of cycles.} Our goal is to obtain comparable (to the growing degree case) sufficient conditions on $G$, such that $G_p$ exhibits a phase transition with respect to its component structure `similar' to that of $G(n,p)$ --- that is, the typical emergence (when the expected average degree is above one) of a unique giant component taking linear number of the vertices, whereas all other components are typically much smaller, of order logarithmic in the number of vertices.

Note that for a fixed $d\ge 3$, given a $d$-regular graph $G$ and $p=\frac{\lambda}{d-1}$ with $\lambda<1$, it is known that \textbf{whp} all components of $G_p$ are of order $O(\log n)$ (see \cite{NP10}). We thus focus our attention on the supercritical regime, that is, when $p=\frac{\lambda}{d-1}$ with $\lambda>1$.

Before stating our main result, let us give some intuition and discuss previous related results. For fixed integer $d\ge 3$ and $p\in (0,1)$, let $q\coloneqq q(p)$ be the unique solution in $(0,1)$ of the equation
\begin{align}\label{eq: extinction}
    q=\left(1-p+pq\right)^{d-1}.
\end{align}
This is the extinction probability of a Galton-Watson tree with offspring distribution $Bin(d-1, p)$ (see, for example, \cite{AN72, D19}). Consider then the probability $y\coloneqq y(p)$ that the root of an infinite $d$-regular tree belongs to an infinite cluster after $p$-bond-percolation. Then, $y$ is given by
\begin{align}\label{eq: survival}
    y=\sum_{i=1}^{d}\binom{d}{i}p^i(1-p)^{d-i}(1-q^i)=1-(1-p+pq)^{d}=1-q(1-p)-pq^2,
\end{align}
where $q=q(p)$ is defined according to \eqref{eq: extinction}. As we will shortly see, and as might be anticipated, $y$ will serve us both in the subsequent discussion and in the statement of our main result. We note that when $d$ tends to infinity and $p=\frac{1+\epsilon}{d-1}$ for small constant $\epsilon>0$, then $y$ is asymptotically equal to the unique solution in $(0,1)$ of $1-y=\exp\{-(1+\epsilon)y\}$, and \textbf{whp} the largest component in supercritical $G(n,p)$ is of order $(1+o(1))yn$ (see, for example, Remark 1.2 of \cite{KLS20}).

A fairly well-studied family of graphs are (percolation on) random $d$-regular graphs. For fixed $d\ge 3$, Goerdt \cite{G01} (see also the results of Alon, Benjamini, and Stacey \cite{ABS04}) showed that for a typical random $d$-regular graph $G$ on $n$ vertices, $G_p$ exhibits a phase transition with respect to the size of its largest component around $p=\frac{1}{d-1}$ (see also \cite{JP18} for behaviour at criticality). In a subsequent work, Pittel \cite{P08} further refined this result: when $p=\frac{1-\epsilon}{d-1}$, for some small constant $\epsilon>0$, \textbf{whp}\footnote{With high probability, that is, with probability tending to one as $n$ tends to infinity.} all components of $G_p$ are of order at most $O\left(\frac{\log n}{\epsilon^2}\right)$, whereas when $p=\frac{1+\epsilon}{d-1}$, \textbf{whp} there is a unique giant component of order $(1+o_n(1))yn$ in $G_p$ (where $y=y(p)$ is defined according to \eqref{eq: survival}), and typically all other components of $G_p$ are of order $O(\log^{1+o(1)}n)$. In the barely supercritical regime, that is when $\epsilon= \epsilon(n)\to 0$ and $\epsilon\gg n^{-1/3}$, more precise results are known --- see \cite{NP10, R12}. There have also been several results considering percolation on graphs chosen uniformly at random from all graphs with a given degree sequence, see \cite{BBCR10, BR15, F07, FJP22, J09, LMP24, R12}.

Another family of graphs which has been studied are high-girth expanders. For fixed $d\ge 3$, Alon, Benjmaini, and Stacey \cite{ABS04} argued that when $G$ is a $d$-regular high-girth expander on $n$ vertices, then $G_p$ exhibits a phase transition with respect to the size of its largest component around $p=\frac{1}{d-1}$: setting $p=\frac{\lambda}{d-1}$, when $\lambda<1$ \textbf{whp} all components of $G_p$ are of order at most $O(\log n)$, whereas when $\lambda>1$ \textbf{whp} there is a \textit{unique} giant component of order $\Theta(n)$. In a subsequent work, Krivelevich, Lubetzky, and Sudakov \cite{KLS20} showed that when $\lambda>1$, for every $\delta, b>0$ there exists $R$ such that if $G$ is a $d$-regular graph with girth at least $R$ and satisfies that for every $U\subseteq V(G)$ with $|U|\le \frac{n}{2}$, $e_G(U,U^c)\ge b|U|$, then \textbf{whp} $\left|1-|L_1|/yn\right|\le \delta$, where $y=y(p)$ is defined according to \eqref{eq: survival} (see also \cite{ABS23} for further generalisations). In terms of the typical sizes of the other components in the supercritical regime, Alon, Benjamini, and Stacey \cite{ABS04} showed that in the supercritical regime, the second largest component is of order $O(n^a)$, for some $a=a(b)\in [0,1)$. Krivelevich, Lubetzky, and Sudakov \cite{KLS20} showed that this is in fact tight: for any $a<1$, there are constant degree high-girth expanders $G$, where the second-largest component of supercritical $G_p$ is in fact of order at least $n^a$ for any $a<1$.

Our main result gives sufficient conditions on a $d$-regular graph $G$, where $d\ge 3$ is fixed, such that when $p=\frac{\lambda}{d-1}$, for some constant $1<\lambda<d-1$, there typically emerges a unique giant component, and all other components are of order at most logarithmic in $|V(G)|$.
\begin{theorem}\label{th: main}
Fix an integer $d\ge 3$. Let $1<\lambda<d-1$ be a constant, and let $p=\frac{\lambda}{d-1}$. Let $b,\alpha>0$ be constants. Then, there exist constants $c\coloneqq c(\lambda, \alpha)>0$ and $C\coloneqq C(d, \lambda, b)>0$ such that the following holds. Let $G$ be a $d$-regular graph on $n$ vertices, satisfying:
\begin{enumerate}[(P\arabic*{})]
    \item for every $U\subseteq V(G)$ with $|U|\le n/2$, $e_G(U,U^c)\ge b|U|$; and, \label{p: global}
    \item for every $U\subseteq V(G)$ with $|U|\le \log^Cn$, $e_G(U)\le (1+c)|U|$. \label{p: local}
\end{enumerate}
Then, \textbf{whp} $G_p$ contains a unique giant component, $L_1$, satisfying $\left|1-\frac{|L_1|}{yn}\right|\le \alpha$, where $y$ is defined according to \eqref{eq: survival}. Furthermore, \textbf{whp} every other component in $G_p$ is of order at most $5\lambda\log n /(\lambda-1)^2$.
\end{theorem} 
Some comments are in place. We note that Assumption \ref{p: global} can be guaranteed by proving that the second largest eigenvalue of the adjacency matrix of $G$ is bounded away from $d$ (see \cite{AM85}). Further, while we treat $b$ as a constant, it follows readily from the proof that one can allow $b$ to be $1/\text{polylog}(n)$ by taking $C$ to be a sufficiently large constant. The aforementioned construction of Krivelevich, Lubetzky, and Sudakov~\cite{KLS20} demonstrates that \textit{some} `local' requirement is indeed necessary in order to guarantee that the second largest component is typically of order $O(\log n)$. In fact, it suffices to require Assumption \ref{p: local} only for connected sets $U\subseteq V(G)$. Observe that any graph $G$ in which every two cycles of length at most $C\log \log n$ are at distance at least $C\log \log n$ satisfies Assumption \ref{p: local} as well. Since a random $d$-regular graph typically satisfies this (see, for example, \cite{W99}), and since a random $d$-regular graph \textbf{whp} satisfies Assumption \ref{p: global} (see \cite{B88}), our results apply to typical random $d$-regular graphs, and can be seen as a generalisation (and improvement) of the results of Pittel in the strictly supercritical regime \cite{P08}. It also follows that expanding $d$-regular $n$-vertex graphs $G$ with girth $\Omega(\log\log n)$ have the behaviour in the supercritical regime postulated by Theorem~\ref{th: main}. Moreover, note that when $\lambda=1+\epsilon$ for sufficiently small constant $\epsilon>0$, we obtain that typically the second largest component is of order $O(\log n/\epsilon^2)$, similarly to the case of supercritical $G(n,p)$. Moreover, when $d$ and $\lambda$ are sufficiently large, we have that \textbf{whp} the second largest component is of order $O(\log n/\lambda)$, similarly to the `very' supercritical regime of $G(n,p)$ with $p=\frac{\lambda}{n}$, where the second largest component is typically of order $O\left(\log n/(\lambda-1-\log \lambda)\right)$. While we treat $\lambda>1$ as a constant, with minor modifications our results also apply to parts of the barely supercritical regime, where $\lambda=1+\epsilon$ and $\epsilon=\epsilon(n)\to 0$.

The paper is structured as follows. In Section \ref{s: prelim} we set out some notation, a modification of the Breadth First Search (BFS) algorithm, and a couple of lemmas which we will use throughout the paper. Section \ref{s: proof} is devoted to the proof of our main result. Finally, in Section \ref{s: discuss} we discuss our results, possible generalisations, and avenues for future research.

\section{Preliminaries}\label{s: prelim} 
Consider a graph $H$ and a vertex $v \in V(H)$. Let $C_H(v)$ be the connected component in $H$ containing $v$. For an integer $r$, let $B_H(v, r)$ be the ball of radius $r$ in $H$ centred at $v$, that is, the set of all vertices of $H$ at a distance of at most $r$ from $v$. For $u, v \in V(H)$, denote the distance in $H$ between $u$ and $v$ by $dist_H(u, v)$. For $S \subseteq V(H)$, set $dist_H(u, S) \coloneqq \min_{v \in S} dist_H(u, v)$. As is fairly standard, $E_H(S, S^c)$ represents the set of edges in $H$ with one endpoint in $S$ and the other in $V(H) \setminus S$, and $E_H(S)$ denotes the set of edges in the induced subgraph $H[S]$. We use $e_H(S, S^c) \coloneqq |E_H(S, S^c)|$ and $e_H(S) \coloneqq |E_H(S)|$. We denote by $N_H(S)$ the external neighbourhood of $S$ in $H$. When the graph is clear from the context, the subscript may be omitted. All the logarithms are with the natural base. Throughout the paper, we systematically ignore rounding signs for the sake of clarity of presentation.

\subsection{A modified Breadth First Search process}\label{s: BFS}
We will utilise the following modification of the classical Breadth First Search (BFS) exploration algorithm. The algorithm receives as an input a graph $H=(V,E)$ with an ordering $\sigma$ on its vertices, a subset $U\subseteq V$, and a sequence $(X_i)_{i=1}^{|E|}$ of i.i.d Bernoulli$(p)$ random variables. 

We maintain three sets throughout the process: $S$, the set of vertices whose exploration has been completed; $Q$, the set of vertices currently being explored, kept in a queue (first-in-first-out discipline); and $T$, the set of vertices which have yet been processed. We initialise $S=\varnothing$, $Q=U$, and $T=V\setminus U$. The process stops once $Q$ is empty.

At round $t$ of the algorithm execution, we consider the first vertex $v$ in $Q$, and query the first edge between $v$ and $T$ according to the order $\sigma$ (that is, the edge connecting $v$ to the first -- according to $\sigma$ -- vertex in $T$). If $X_t=1$, we retain this edge and move its endpoint vertex in $T$ to $Q$. If $X_t=0$, we discard the edge and continue. If there are no neighbours of $v$ in $T$ left to query, we move $v$ from $Q$ to $S$ and continue. 

Note that once $Q$ is empty, $H_p[S]$ has the same distribution as $\bigcup_{u\in U}C_{H_p}(u)$. Moreover, at every moment we have queried (and answered in the negative) all edges between current $S$ and $T$.

\subsection{Auxiliary Lemmas}
We will make use of the following two fairly standard probability bounds. The first one is a Chernoff-type tail bound for the binomial distribution (see, for example, Appendix A in \cite{AS16}).
\begin{lemma}\label{l:  chernoff}
Let $n\in \mathbb{N}$, let $p\in [0,1]$, and let $X\sim Bin(n,p)$. Then for any $0<t\le \frac{np}{2}$, 
\begin{align*}
    &\mathbb{P}\left[|X-np|\ge t\right]\le 2\exp\left\{-\frac{t^2}{3np}\right\}.
\end{align*}
%Further, for any $b>0$,
%\begin{align*}
%    \mathbb{P}\left[X\ge bnp\right]\le \left(\frac{e}{b}\right)^{bnp}.
%\end{align*}
\end{lemma}

The second one is a variant of the well-known Azuma-Hoeffding inequality (see, for example, Chapter 7 in \cite{AS16}).
\begin{lemma}\label{l: azuma}
Let $m\in \mathbb{N}$ and let $p\in [0,1]$. Let $X = (X_1,X_2,\ldots, X_m)$ be a random vector with range $\Lambda = \{0,1\}^m$ with each $X_{\ell}$ distributed according to independent Bernoulli$(p)$. Let $f:\Lambda\to\mathbb{R}$ be such that there exists $C \in \mathbb{R}$ such that for every $x,x' \in \Lambda$ which differ only in one coordinate,
\begin{align*}
    |f(x)-f(x')|\le C.
\end{align*}
Then, for every $t\ge 0$,
\begin{align*}
    \mathbb{P}\left[\big|f(X)-\mathbb{E}\left[f(X)\right]\big|\ge t\right]\le 2\exp\left\{-\frac{t^2}{2C^2m}\right\}.
\end{align*}
\end{lemma}

%(Sahar) : The next one is not really necessary, but it might allow us to "discard" some dependencies on d, which is nice.
%\subsection{Deterministic properties of the graph}
We will also use the following lemma, relating our local density assumption to local vertex expansion.
\begin{lemma}\label{l: assumption to vtx expanson}
Let $d\ge 3$ and let $G$ be a $d$-regular graph on $n$ vertices satisfying Assumption \ref{p: local} with $C> 1$ and $0<c<\frac{1}{4}$. Then, for every $v\in V(G)$,
we have $|B(v,2\log\log n)|\ge \frac{10\lambda\log n}{(\lambda-1)^2}$.
\end{lemma}
\begin{proof}
Fix $r\in \mathbb{N}$ and let $U\coloneqq B(v,r)$. Note that $U\cup N(U)=B(v,r+1)$. If $|U\cup N(u)|\le \log^Cn$, then by \ref{p: local}, $e(U\cup N(U))\le (1+c)|U\cup N(U)|$. On the other hand, again by \ref{p: local}, $e(U\cup N(U))\ge d|U|-e(U)\ge (d-1-c)|U|$. Hence, $|U\cup N(U)|\ge \frac{|U|(d-1-c)}{1+c}$. Thus, for every $r\in \mathbb{N}$,
\begin{align*}
    |B(v,r+1)|\ge \min\left\{\log^Cn, \left(1+\frac{d-1-c}{1+c}\right)|B(v,r)|\right\}.
\end{align*}
Therefore, taking $r_0=\log_{1+\frac{d-1-c}{1+c}}\left(\frac{10\lambda\log n}{(\lambda-1)^2}\right)$, we have that $|B(v,r_0)|\ge \frac{10\lambda\log n}{(\lambda-1)^2}$. All that is left is to note that $r_0\le 2\log\log n$.
\end{proof}

Finally, we will utilise the following lemma, showing that if $G$ is a $d$-regular graph satisfying Assumption \ref{p: local}, then there are `many' vertices in $G$ which are far from any cycle. Formally,
\begin{lemma}\label{l: far from cycles}
Let $d\ge 3$ and let $G$ be a $d$-regular graph on $n$ vertices satisfying Assumption \ref{p: local} with a sufficiently small constant $c>0$. Then, there is a set $X\subseteq V(G)$ with $|X|\ge \left(1-\frac{1}{(d-1)^{\frac{1}{16c}}}\right)n$ such that for every $v\in X$, there are no cycles in $B\left(v,\frac{1}{16c}\right)$.  
\end{lemma}
\begin{proof}
We first claim that in $G$, every two cycles of length at most $\frac{1}{4c}$ are at distance at least $\frac{1}{4c}$. Otherwise, let $U$ be the vertices of these two cycles, together with the vertices of a path of length at most $\frac{1}{4c}$ connecting them. Then, $|U|\le 2\cdot \frac{1}{4c}+\frac{1}{4c}=\frac{3}{4c}$ and $e(U)\ge |U|+1$ (since $U$ is connected and contains at least two cycles). On the other hand, by \ref{p: global},
\begin{align*}
    e(U)\le (1+c)|U|\le |U|+c\cdot \frac{3}{4c}<|U|+1.
\end{align*}

Let $m$ be the number of cycles of length at most $\frac{1}{4c}$ in $G$, denote them by $C_1,\ldots, C_m$. For every $i\in [m]$, let $N_i$ be the set of vertices in $V(G)$ at distance at most $\frac{1}{8c}$ from $C_i$. We then have that for every $i\neq j\in [m]$, $N_i\cap N_j=\varnothing$ (as otherwise we would have two cycles of length at most $\frac{1}{4c}$ at distance at most $2\cdot\frac{1}{8c}=\frac{1}{4c}$). By the same reasoning, for every $i\in [m]$, $G[N_i]$ has only one cycle. We thus obtain that $m\cdot (d-1)^{\frac{1}{8c}}\le n$, that is, $m\le \frac{n}{(d-1)^{1/(8c)}}$. 

Since $m(d-1)^{1/(16c)}\le \frac{n}{(d-1)^{1/(16c)}}$, we conclude that there are at most $\frac{n}{(d-1)^{1/(16c)}}$ vertices in $G$ which are at distance at most $\frac{1}{16c}$ from a cycle of length at most $\frac{1}{4c}$. Hence there is a set $X\subseteq V(G)$, $|X|\ge \left(1-\frac{1}{(d-1)^{1/(16c)}}\right)n$, such that every $v\in X$ is at distance at least $\frac{1}{16c}$ from any cycle of length at most $\frac{1}{4c}$. Thus, every $v\in X$ satisfies that $B\left(v,\frac{1}{16c}\right)$ contains no cycles.
\end{proof}

\section{Proof of Theorem \ref{th: main}}\label{s: proof}
Throughout the section, we assume $C>0$ is a large enough constant with respect to $d,\lambda-1$ and $b$, and that $c>0$ is a small enough constant with respect to $\lambda-1$ and $\alpha$. 

We utilise a double-exposure/sprinkling argument \'a la Ajtai-Koml\'os-Szemer\'edi~\cite{AKS81}. Let $\delta\coloneqq \delta(\lambda, d)>0$ be a sufficiently small constant, satisfying that $\lambda-\delta>1$. Let $p_2=\frac{\delta}{d-1}$ and let $p_1$ be such that $(1-p_1)(1-p_2)=1-p$, noting that $p_1\ge \frac{\lambda-\delta}{d-1}$ and that $G_p$ has the same distribution as $G_{p_1}\cup G_{p_2}$. We abbreviate $G_1\coloneqq G_{p_1}$ and $G_2\coloneqq G_{p_1}\cup G_{p_2}$. The overall strategy will be similar to that of the growing degree case \cite{DK25B}, however the analysis itself (and where the difficulties lie) is quite different. We begin by considering `large' components in $G_1$. We show that typically there are no components in $G_1$ (nor in $G_2$) whose size is between $5\lambda\log n/(\lambda-1)^2$ and $\log^Cn$ (Lemma \ref{l: gap statement}). Here, unlike the growing degree case, this does not follow from a first moment argument, and we will utilise the BFS algorithm in order to estimate the probability a fixed vertex lies in a component of a given size. Utilising our modified BFS algorithm, we further show that \textbf{whp} in $G_1$, every vertex is within distance $O(\log\log n)$ from a `large' component (Lemma \ref{l: large components are well spread}). Then, in a manner similar to that of \cite{AKS81}, using these two properties, we show that typically after sprinkling with probability $p_2$, all `large' components in $G_1$ merge into a unique component $L_1$ in $G_2$ (Lemma \ref{l: merge}). Finally, we show that when sprinkling with probability $p_2$, we typically do not `accidentally' merge small components (of order $O(\log n)$) in $G_1$ into a large component (of order $\Omega(\log^C n)$) in $G_2$ (Lemma \ref{l: uniqueness}). All that is then left is to argue about the size of $L_1$, which we do by utilising Lemma \ref{l: far from cycles}, in a fashion somewhat similar to \cite{KLS20}.

More formally, let $W$ be the set of vertices in `large' components in $G_1$, that is,
\begin{align*}
    W\coloneqq \left\{v\in V(G)\colon |C_{G_1}(v)|\ge \frac{5\lambda\log n}{(\lambda-1)^2}\right\}.
\end{align*}

We begin by showing that \textbf{whp} there are no components in $G_p$ whose order is between $\frac{5\lambda\log n}{(\lambda-1)^2}$ and $\log^Cn$. 
\begin{lemma}\label{l: gap statement}
Fix $v\in V(G)$, and let $k\in \left[\frac{1}{16c},\log^Cn\right]$. The probability that $|C_{G_p}(v)|=k$ is at most $\exp\left\{-\frac{(\lambda-1)^2k}{4\lambda}\right\}$. In particular, there are no components in $G_p$ whose order is between $\frac{5\lambda\log n}{(\lambda-1)^2}$ and $\log^C n$. 
\end{lemma}
\begin{proof}
We will utilise the BFS algorithm described in Section \ref{s: BFS}, with $U=\{v\}$ (that is, we initialise $Q=\{v\}$), with $H=G$ and with probability $p$. 

Fix $k\in [1/(16c),\log^Cn]$. If $|C_{G_p}(v)|=k$, then there is some moment $t$ where $Q$ is empty and $|S|=k$. Since $k\le \log^Cn$, by Assumption \ref{p: local}, $e_G(S)\le (1+c)k$, and since $G$ is $d$-regular and $Q=\varnothing$, $e_G(S,T)=e_G(S,S^c)\ge (d-2-2c)k$. Hence, we have had at least $(d-2-2c)k$ queries corresponding to $e_G(S,T)$, and at least additional $k-1$ queries corresponding to the internal edges of $C_{G_p}(v)$ which have been explored. Further, only $k-1$ queries were answered in the positive. Therefore, by Lemma \ref{l:  chernoff},
\begin{align*}
    \mathbb{P}\left[|C_{G_p}(v)\right]&\le \mathbb{P}\left[Bin\left(k(d-1-2c)-1,\frac{\lambda}{d-1}\right)\le k-1\right]\\
    &\le \exp\left\{-\frac{(\lambda-1)^2k^2}{4\lambda k}\right\}=\exp\left\{-\frac{(\lambda-1)^2k}{4\lambda}\right\},
\end{align*}
where we assumed that $c$ is sufficiently small with respect to $\lambda-1$. 

As for the second part of the lemma's statement, for $k\in [5\lambda\log n/(\lambda-1)^2,\log^Cn]$ the above implies the probability that a fixed $v$ is in a component of order $k$ is at most $\exp\left\{-\frac{(\lambda-1)^2k}{4\lambda}\right\}=o\left(\frac{1}{n\log^Cn}\right)$. Therefore, the union bound over the $n$ possible choices of $v$ and $\log^Cn$ choices of $k$ completes the proof.
\end{proof}
Note that by our assumptions on $p_1$, the above holds (with the same proof) in $G_1$.

We now turn to show that typically every vertex is within distance $O(\log\log n)$ from a vertex in $W$.
\begin{lemma}\label{l: large components are well spread}
\textbf{Whp}, for every $v\in V(G)$ we have that $dist_{G}(u,W)\le 2\log\log n$.
\end{lemma}
\begin{proof}
Fix $v\in V(G)$. By Lemma \ref{l: assumption to vtx expanson}, we have that $|B(v,2\log\log n)|\ge \frac{10\lambda\log n}{(\lambda-1)^2}$. Let $Y_v$ be an arbitrary set of $\frac{10\lambda\log n}{(\lambda-1)^2}$ vertices in $B(v,2\log\log n)$. We now run the BFS algorithm described in Section \ref{s: BFS} with $U=Y_v$ (that is, we initialise $Q=Y_v$), $H=G$, and probability $p_1$. 

Suppose towards contradiction that at the moment $t$ when $Q$ emptied, $|S|\coloneqq s\le \log^Cn$. Then, all the edges between $S$ and $T=V(G)\setminus S$ have been queried and answered in the negative. By Assumption \ref{p: local}, we have that $e_G(S,T)=e_G(S,S^c)\ge (d-2-2c)s$. Hence, $t\ge (d-2-2c)s+s-\frac{10\lambda\log n}{(\lambda-1)^2}$, and we have received $s-\frac{10\lambda\log n}{(\lambda-1)^2}$ positive answers. By Lemma \ref{l:  chernoff}, the probability of this event is at most
\begin{align*}
    \mathbb{P}\left[Bin\left((d-1-2c)s-\frac{10\lambda\log n}{(\lambda-1)^2},\frac{\lambda-\delta}{d-1}\right)\le s-\frac{10\lambda\log n}{(\lambda-1)^2}\right]\le \exp\left\{-\frac{((\lambda-1) s/2)^2}{4\lambda s}\right\}=o(1/n),
\end{align*}
where we used our assumption that $c,\delta$ are sufficiently small with respect to $\lambda-1$, and the fact that $s\ge \frac{10\lambda\log n}{(\lambda-1)^2}$. 

Thus, by the union bound over the $n$ possible choices of $v$, we have that \textbf{whp} for every $v\in V(G)$, there is a set $Y_v\subseteq B(v,2\log\log n)$ of order $\frac{10\lambda\log n}{(\lambda-1)^2}$, such that
\begin{align*}
    \left|\bigcup_{u\in Y_v}C_{G_1}(u)\right|\ge \log^Cn.
\end{align*}
Assuming that $C>3$, we conclude that \textbf{whp} there is a component of order $\Omega(\log^2n)\ge \frac{5\lambda\log n}{(\lambda-1)^2}$ in $G_1$ at distance at most $2\log\log n$ from every vertex $v\in V(G)$.
\end{proof}

We are now ready to argue that after sprinkling with probability $p_2$, \textbf{whp} all components in $W$ merge.
\begin{lemma}\label{l: merge}
\textbf{Whp} there is a component $K$ in $G_2$, such that $W\subseteq V(K)$.
\end{lemma}
\begin{proof}
By Lemma \ref{l: gap statement}, \textbf{whp} there are no components in $G_1$ whose size is between $\frac{5\lambda\log n}{(\lambda-1)^2}$ and $\log^Cn$. Thus, \textbf{whp} every component in $G_1[W]$ is of order at least $\log^Cn$. Further, by Lemma \ref{l: large components are well spread}, \textbf{whp} every $v\in V(G)$ is at distance at most $2\log\log n$ from some $w\in W$. We continue assuming these properties hold deterministically.

It suffices to show that \textbf{whp} for every partition of $W$ into two $G_1$-component-respecting parts $A$ and $B$, with $a=|A|\le |B|$, there exists a path in $G_{p_2}$ between $A$ and $B$. To that end, let $A'$ be $A$ together with all the vertices in $V(G)\setminus B$ which are at distance at most $2\log \log n$ from some vertex in $A$. Similarly, let $B'$ be $B$ together with all the vertices in $V(G)\setminus A'$ which are at distance at most $2\log\log n$ from some vertex in $B$. Note that $V(G)=A'\sqcup B'$, and thus by \ref{p: global}, $e_G(A',B')=e_G(A',V(G)\setminus A')\ge b|A|$. We can very crudely extend these edges into at least $\frac{b\cdot a}{d^{5\log\log n}}$ edge-disjoint paths (in $G$) of length at most $4\log\log n$ between $A$ and $B$ (indeed, every edge belongs to at most $d^{4\log\log n}\cdot 4\log\log n<d^{5\log\log n}$ paths of length at most $4\log\log n$). Since every component of $W$ is of size at least $\log^Cn$ and since $W=A\sqcup B$ is a component-respecting partition, given that $|A|=a$ there are at most $\sum_{i=1}^{a/\log^Cn}\binom{n/\log^Cn}{a/\log^Cn}\le n^{a/\log^Cn}$ ways to choose $A$ (and hence the partition). Thus, by the union bound, the probability there exists such a partition without a path in $G_{p_2}$ between $A$ and $B$ is at most
\begin{align*}
    \sum_{a=\log^Cn}^{n/2}n^{a/\log^Cn}(1-p_2^{4\log\log n})^{\frac{b\cdot a}{d^{5\log\log n}}}&\le \sum_{a=\log^Cn}^{n/2}n^{a/\log^Cn}\exp\left\{-ab\left(p_2/d\right)^{5\log\log n}\right\}\\
    &\le \sum_{a=\log^Cn}^{n/2} \exp\left\{a\left(\frac{1}{\log^{C-1}n}-\frac{b}{\log^{10\log d+5\log(1/\delta)}n}\right)\right\}\\
    &\le n\cdot \exp\left\{-\log^Cn\frac{b}{2\log^{10\log d+5\log(1/\delta)}n}\right\}=o(1),
\end{align*}
where we assumed that $C$ is sufficiently large with respect to $b, d$, and $\lambda-1$.
\end{proof}

By Lemma \ref{l: gap statement}, \textbf{whp} there are no components in $G_1$, nor in $G_2$, whose order is between $5\lambda\log n/(\lambda-1)^2$ and $\log^Cn$. Further, by Lemma \ref{l: merge} \textbf{whp} all components in $G_1$ whose order was at least $5\lambda\log n/(\lambda-1)^2$ merged into a unique component. Note, however, that we still need to rule out the existence of components outside of $W$ whose order is at least $\log^Cn$ in $G_2$: any such component would be composed of components in $G_1$ whose size is at most $5\lambda\log n/(\lambda-1)^2$. We give two proofs for this, showcasing different approaches and arguments. 
\begin{lemma}\label{l: uniqueness}
\textbf{Whp}, there is no component in $G_2$ of order at least $\log^Cn$, which does not intersect $W$.
\end{lemma}
\begin{proof}
We begin by exposing $G_1$. Let us show that after sprinkling with probability $p_2$, there is no component in $G_2$ of order at least $\log^Cn$ which does not intersect $W$. 

%Let $M$ be a connected (in $G$) set of size at least $\log^Cn$ satisfying $M\cap W=\varnothing$. We can then find a connected subset $M'\subseteq M$, whose size lies in the interval $[\log^2 n, 2\log^2 n]$, so that $E_{G_1}[M', V(G)\setminus M']=\varnothing$ (indeed, every component of $G_1[V(G)\setminus W]$ is of size at most $5\lambda\log n/(\lambda-1)^2$, and we can form a connected subset $M'$ by joining these components one by one until the first time we reach at least $\log^2n$ vertices, and have that $|M'|\le \log^2n+5\lambda\log n/(\lambda-1)^2<2\log^2n$).

Suppose first that $G_1$ is such that one can create a connected set $M$ satisfying $|M|\in [\log^2n, 2\log^2n]$, $M\cap W=\varnothing$ and $E_{G_1}[M, V(G)\setminus M]=\varnothing$ by adding at most $t=\frac{(\lambda-1)^2\log^2n}{5\lambda\log((d-1)/\delta)}$ edges to $G_1$. Then, the probability that $G_2$ contains a connected component $K$ whose size is in the interval $[\log^2n, 2\log^2n]$ is at least
\begin{align*}
    p_2^t\left(1-p_2\right)^{d\cdot 2\log^2n}&\ge \exp\left\{-t\log\left(\frac{d-1}{\delta}\right)-\delta\cdot 4\log^2n\right\}\\
    &\ge \exp\left\{-\frac{\left((\lambda-1)^2+20\delta\lambda\right)\log^2n}{5\lambda}\right\}.
\end{align*}
On the other hand, by Lemma \ref{l: gap statement}, the probability there is a connected component in $G_2$ whose order lies in the interval $[\log^2n, 2\log^2n]$ is at most
\begin{align*}
    n\cdot 2\log^2n \cdot \exp\left\{-\frac{(\lambda-1)^2\log^2n}{4\lambda}\right\},
\end{align*}
which is a contradiction since $\frac{(\lambda-1)^2\log^2n}{4\lambda}-2>\frac{\left((\lambda-1)^2+20\delta\lambda\right)\log^2n}{5\lambda}$, where we assumed that $\delta$ is sufficiently small with respect to $\lambda-1$.

We can thus assume that in $G_1$, in order to create a connected set $M$ satisfying $|M|\in [\log^2n, 2\log^2n]$, $M\cap W=\varnothing$ and $E_{G_1}[M, V(G)\setminus M]=\varnothing$, we must add at least $t$ edges. Utilising that, let us now show that \textbf{whp} when exploring the connected component in $G_2$ of any component $K$ of $G_1[V(G)\setminus W]$, we could not uncover a component of order at least $\log^Cn$ which does not intersect $W$. 

To that end, consider the following variant of the BFS algorithm, which receives as input the graph $G$ with an order $\sigma$ on its vertices, the subgraph $G_1\subseteq G$, a component $K$ in $G_1[V(G)\setminus W]$, and a sequence $(X_i)_{i=1}^{nd/2}$ of independent Bernoulli$(p_2)$ random variables. As in Section \ref{s: BFS}, the algorithm maintains three sets: $S$, the set of vertices whose exploration has been completed; $Q$, the set of vertices currently being explored, kept in a queue (first-in-first-out discipline); and $T$, the set of vertices which have yet been processed. We initialise $S=\varnothing$, $Q=V(K)$, and $T=V\setminus (K\cup W)$. The process stops once $Q$ is empty.

At round $\tau$, we consider the first vertex $v$ in $Q$, and query the first edge from $v$ to $T$ according to the order $\sigma$, denote its endpoint in $T$ by $u$. If $X_\tau=1$, we retain this edge and move the set $C_{G_1}(u)$ from $T$ to $Q$. If $X_t=0$, we discard the edge and continue. If there are no neighbours of $v$ in $T$ left to query, we move $v$ from $Q$ to $S$ and continue. 

Note that at the end of this process, the component of $G_2[V(G)\setminus W]$ intersecting with $K$ has the same distribution as $G_2[S]$. If the component of $G_2[V(G)\setminus W]$ intersecting with $K$ is of order at least $\log^Cn$, then must have been a moment $\tau$ where $|S\cup Q|\in [\log^2n, 2\log^2n]$. Indeed, every component of $G_1[V(G)\setminus W]$ is of size at most $5\lambda\log n/(\lambda-1)^2$, and at the first time when $|S\cup Q|\ge \log^2n$, we have that $|S\cup Q|\le \log^2n+5\lambda\log n/(\lambda-1)^2<2\log^2n$. Thus, $\tau \le 2d\log^2n$. On the other hand, by our assumption, by moment $\tau$ we received at least $t$ positive answers. By the union bound over the at most $n$ possible choices of $K$, and by Lemma \ref{l:  chernoff}, the probability of this event is at most
\begin{align*}
    n\cdot \mathbb{P}\left[Bin\left(2d\log^2n, \frac{\delta}{d-1}\right)\ge \frac{(\lambda-1)^2\log^2n}{5\lambda\log((d-1)/\delta)}\right]\le n\cdot 2\exp\left\{-\Omega(\log^2n)\right\}=o(1),
\end{align*}
where we assumed that $\delta$ is small enough with respect to $\lambda-1$ and $d$.
\end{proof}

Let us also give a second proof, which is shorter, and somewhat similar to the BFS-type arguments given previously in this paper.
\begin{proof}[Alternative proof of Lemma \ref{l: uniqueness}]
For $v\in V(G)$, let $\mathcal{A}_v$ be the event that $C_{G_2}(v)\cap W=\varnothing$ and \break ${|C_{G_2}(v)|\ge \log^2n}$. By Lemma \ref{l: gap statement}, we have that the probability of an event violating the statement of the Lemma is at most $\mathbb{P}\left[\bigcup_{v\in V(G)}\mathcal{A}_v\right]+o(1)$.

Let $(X_i)_{i=1}^{|E|}$ be a sequence of i.i.d Bernoulli($p_1$) random variables, and let $(Y_i)_{i=1}^{|E|}$ be a sequence of i.i.d. Bernoulli $(p_2)$ random variables. We will run a BFS-type algorithm, very similar to the one in Section \ref{s: BFS}, however, we will utilise both sequences of random variables.

Fix $v\in V(G)$. We initialise $Q=\{v\}$. Now, as long as $|S\cup Q|\le \log^2n$, when we query the $t$-th edge, we retain it if $X_t=1$ or $Y_t=1$, and discard it only if $X_t=Y_t=0$. Once $|S\cup Q|=\log^2n$ (and at any subsequent moment), when we query the $t$-th edge, we retain it if $X_t=1$, and discard it if $X_t=0$. Note that in this manner, up until $|S\cup Q|=\log^2n$, the exploration of the component of $v$ is in $G_2$, and afterwards we continue the exploration in $G_1$.  Also, each query obtains a positive answer independently and with probability at least $p_1$. 

Note that if $\mathcal{A}_v$ occurs, then we must have reached a moment where $|S\cup Q|=\log^2n$. Suppose that $C_{G_2}(v)\cap W=\varnothing$. Then, the above process must have ended before $|S\cup Q|=\log^4n$. Indeed, otherwise, there are vertices in $C_{G_2}(v)$ belonging to components of size at least $\log^4n/\log^2n=\log^2n$ in $G_1$, and thus (by definition) intersect with $W$. Hence, the probability that $\mathcal{A}_v$ occurs is at most the probability that the above process stopped at some moment $t$ where $\log^2n\le |S|\le \log^4n$. Let $k\coloneqq |S|$. At that moment, we had exactly $k-1$ positive answers in the algorithm's run. Moreover, assuming that $C\ge 4$, by Property \ref{p: local} we had additional $e(S,T)=e(S,S^c)\ge (d-2-2c)k$ queries, all answered in the negative. Hence, we had at least $(d-1-2c)k-1$ queries and received only $k-1$ positive answers. Noting that $\left((d-1-2c)k-1\right)\cdot p_1>k-1$, we have by Lemma \ref{l:  chernoff} that
\begin{align*}
    \mathbb{P}\left[\mathcal{A}_v\right]&\le \mathbb{P}\left[Bin\left((d-1-2c)k-1,p_1\right)\le k-1\right]\\
    &\le \exp\left\{-\Theta(k)\right\}=o(1/n),
\end{align*}
where we used that $k\ge \log^2n$. Thus, by the union bound, $\mathbb{P}\left[\bigcup_{v\in V(G)}\mathcal{A}_v\right]=o(1)$, completing the proof.
\end{proof}

We are now ready to prove Theorem \ref{th: main}.
\begin{proof}[Proof of Theorem \ref{th: main}]
By Lemma \ref{l: merge}, \textbf{whp} there is a unique component $L_1$ in $G_2$, such that $W\subseteq V(L_1)$. By Lemma \ref{l: gap statement}, \textbf{whp} any component in $G_2$ besides $L_1$ is either of size at most $5\lambda\log n/(\lambda-1)^2$, or of size at least $\log^Cn$. By Lemma \ref{l: uniqueness}, \textbf{whp} any component in $G_2$ whose size is at least $\log^Cn$ intersects with $W$, and is thus part of $L_1$. Thus, \textbf{whp} all components of $G_2$ besides $L_1$ are of order at most $5\lambda\log n/(\lambda-1)^2$. 

Let us now show that \textbf{whp} $\left|1-\frac{|L_1|}{yn}\right|\le \alpha$, where $y=y(p)$ is defined according to \eqref{eq: extinction}. 

The probability a vertex belongs to a component of order at least $\log^Cn$ in $G_p$ is stochastically dominated by the probability that the root of an infinite $d$-regular tree belongs to an infinite cluster after $p$-bond-percolation. Thus, by standard results (see, for example, \cite{D19}) $|L_1|\le (1+o(1))yn$.

Let $Z_1$ be the random variable counting the number of vertices in components of order at least $\frac{1}{16c}$ in $G_2$. By Lemma \ref{l: far from cycles}, there exists a set $X\subseteq V(G)$, $|X|\ge \left(1-\frac{1}{(d-1)^{\frac{1}{16c}}}\right)n$ such that for every $v\in X$, there are no cycles in $B\left(v,\frac{1}{16c}\right)$. Hence, by standard results, we have that for every $v\in X$, $\mathbb{P}\left[|C_{G_2}(v)|\ge \frac{1}{16c}\right]\ge (1-o_c(1))y$, where $o_c(1)$ tends to zero as $c$ tends to zero. Thus, $\mathbb{E}[|Z_1|]\ge (1+o_c(1))y\cdot \left(1-\frac{1}{(d-1)^{\frac{1}{16c}}}\right)n=(1-o_c(1))yn$. To show that $|Z_1|$ is well concentrated around its mean, consider the standard edge-exposure martingale. Every edge can change the value of $|Z_1|$ by at most $\frac{1}{8c}$. Hence, by Lemma \ref{l: azuma},
\begin{align*}
    \mathbb{P}\left[\left||Z_1|-\mathbb{E}[|Z_1|]\right|\ge n^{2/3}\right]\le 2\exp\left\{-\frac{n^{4/3}}{2\cdot \frac{nd}{2}}\cdot \frac{1}{64c^2}\right\}=o(1).
\end{align*}
Therefore, \textbf{whp}, $|Z_1|\ge (1-o_c(1))yn$. 

Let $Z_2$ be the random variable counting the number of vertices in components of $G_2$ whose order lies in the interval $\left[\frac{1}{16c},\log^Cn\right]$. By Lemma \ref{l: gap statement}, the probability $v\in V(G)$ belongs to such a component is at most $\sum_{k=1/(16c)}^{\log^Cn}\exp\left\{-\frac{(\lambda-1)^2k}{4}\right\}\le \exp\left\{-\frac{(\lambda-1)^2}{70\lambda\cdot c}\right\}$. Thus $\mathbb{E}[|Z_2|]\le o_c(1)n$, where we assumed that $c$ is sufficiently small with respect to $\lambda-1$. Once again, let us consider the standard edge-exposure martingale. Every edge can change the value of $|Z_2|$ by at most $2\log^Cn$. Hence, by Lemma \ref{l: azuma},
\begin{align*}
    \mathbb{P}\left[\left||Z_2|-\mathbb{E}[|Z_2|]\right|\ge n^{2/3}\right]\le 2\exp\left\{-\frac{n^{4/3}}{2\cdot \frac{nd}{2}\cdot 4\log^{2C}n}\right\}=o(1).
\end{align*}

Therefore, we obtain that the number of vertices in components of order at least $\log^Cn$ in $G_2$ is \textbf{whp} $(1-o_c(1))yn-(1+o(1))o_c(1)n$. Thus, given $\alpha$ we can choose $c$ small enough such that \textbf{whp} there are at least $(y-\alpha)n$ vertices in components of order at least $\log^Cn$ in $G_2$. By Lemma \ref{l: uniqueness} every component of order $\log^Cn$ in $G_2$ intersects with $W$. Thus, the number of vertices in components that intersect with $W$ is at least $(y-\alpha)n$. By Lemma \ref{l: merge}, all the vertices in $W$ merge into a unique component $L_1$, and hence \textbf{whp} $|L_1|\ge (y-\alpha)n$. Altogether, we have that $\left|1-\frac{|L_1|}{yn}\right|\le \alpha$.
\end{proof}

\section{Discussion}\label{s: discuss}
We showed that for a fixed $d\ge 3$, any $d$-regular $n$-vertex graph $G$ which satisfies a fairly mild global expansion assumption \ref{p: global}, and does not have dense sets of polylogarithmic order (Property \ref{p: local}), exhibits a phase transition around $p=\frac{1}{d-1}$ similar to that of the binomial random graph $G(n,p)$ around $p=\frac{1}{n}$ (and alike to that of percolation on a random $d$-regular graph) --- that is, the typical emergence of a unique giant component $L_1$ of order linear in $n$, while \textbf{whp} all other components are of order at most logarithmic in $n$.

With slight adaptation, it follows from the above proof that for a fixed $d\ge 3$, given an $n$-vertex graph $G$ with \textit{minimum degree} $d$, when $p=\frac{\lambda}{d-1}$ with $\lambda>1$, \textbf{whp} $G_p$ contains a giant component of order linear in $n$, and all the other components are of order logarithmic in $n$. Indeed, if we rely on the alternative proof given for Lemma \ref{l: uniqueness}, the only place where we use an upper bound on the maximum degree of $G$ is in Lemma \ref{l: merge} (when estimating the number of edge-disjoint paths between $A$ and $B$). To overcome that, one can take the more delicate approach as given in \cite{DK25B}. First, we show that \textbf{whp} for every $v\in V(G)$, a constant fraction of the vertices in $B(v,2\log\log n)$ are in $W$. Using that, we can show that given a $G_1$-component-respecting partition $A\sqcup B$ of $W$, \textbf{whp} we can construct an $\Omega(d)$-regular tree of depth at most $2\log\log n$ from every $v\notin W$, whose leaves are either in $A$ or in $B$. Then, when sprinkling with probability $p_2$ on these trees, we can show that \textbf{whp} for any such partition, there is a path between $A$ and $B$.

We note that while one might anticipate the same could hold if we consider $p\ge (1+\epsilon)p_c(G)$, where $p_c(G)$ is the critical probability of the graph $G$ and $\epsilon>0$ is a small constant, this is quite false. Indeed, consider the graph $G$ constructed as follows (where we note that the construction is similar to the construction given in \cite{KLS20}). Let $0<a<1$ be a constant, and let $H$ be a $100d$-regular graph on $n-\log n \cdot n^{a}$ vertices, which is a fairly good expander and is locally sparse (for example, a random $100d$-regular graph on $n-\log n\cdot n^{a}$ typically satisfies this). Let $T$ be a $C'd$-regular tree (for some sufficiently large $C'>0$) of depth $a\log_{C'd-1}n/2$, and from each leaf of $T$, let us grow a $d$-regular tree on about $n^{a/2}$ vertices. Denote by $T'$ the obtained graph, and note that it has $\approx n^{a/2}\cdot (C'd-1)^{a\log_{C'd-1}n/2}=n^a$ vertices. We now take $\log n$ copies of $T'$, and identify the leaves of each copy with distinct vertices in $H$. Let $G$ be the obtained graph. Note that $G$ is also a good expander and locally sparse. We then have that $p_c(G)=\frac{1+o(1)}{100(d-1)}$, and when $p=(1+\epsilon)p_c(G)$, for large enough $C'$ the second-largest component of $G_p$ would be of order $\Omega(n^{a/2})$.

As mentioned in the introduction, Krivelevich, Lubetzky, and Sudakov \cite{KLS20} showed the existence of $d$-regular graphs on $n$ vertices, satisfying Assumption \ref{p: global} and with large (yet constant) girth, such that the second largest component in the supercritical regime is typically of order polynomial in $n$. Thus, while it is clear that some additional requirement (to \ref{p: global}) is necessary in order to ensure that $G$ exhibits phase transition similar to that of $G(n,p)$, it remains open whether Assumption \ref{p: local} is (essentially) tight, or could it be relaxed --- can we ensure that $G$ exhibits phase transition similar to that of $G(n,p)$, replacing our local density assumption with one that holds for sets up to size, say, $O(\log n)$? Furthermore, as mentioned in the discussion after Theorem \ref{th: main}, it suffices for $G$ to satisfy \ref{p: global} and have girth of order $\Omega(\log\log n)$ in order to ensure such a phase transition. Can this assumption on the girth of $G$ be further relaxed? What assumption on the girth of $G$ would suffice?

\bibliographystyle{abbrv}
\bibliography{perc} 
\end{document}